\documentclass{amsart}
\usepackage{amsmath}
\usepackage{amssymb}
\usepackage{amscd}
\usepackage[plainpages]{hyperref}

\newtheorem{theorem}{Theorem}
\newtheorem{lemma}[theorem]{Lemma}
\newtheorem{proposition}[theorem]{Proposition} 

\theoremstyle{definition}
\newtheorem{remark}[theorem]{Remark}

\newcommand{\A}{\ensuremath{\mathcal A}}
\newcommand{\CP}{\ensuremath{\mathbb{CP}}}
\newcommand{\C}{\ensuremath{\mathbb C}}
\newcommand{\Z}{\ensuremath{\mathbb Z}}
\newcommand{\bS}{\ensuremath{\mathbb S}}
\newcommand{\id}{\operatorname{id}}
\newcommand{\tv}{\operatorname{tv}}
\newcommand{\Mod}{\operatorname{Mod}}

\newcommand{\Aut}{\operatorname{Aut}}
\renewcommand{\bar}{\overline}

\begin{document}
\title[Automorphism groups of some pure braid groups]{Automorphism groups of some pure braid groups}
\author[Daniel C. Cohen]{Daniel C. Cohen$^\dag$}
\address{Department of Mathematics, Louisiana State University, Baton Rouge, LA 70803}
\email{\href{mailto:cohen@math.lsu.edu}{cohen@math.lsu.edu}}
\urladdr{\href{http://www.math.lsu.edu/~cohen/}
{www.math.lsu.edu/\char'176cohen}}
\thanks{{$^\dag$}Partially supported by Louisiana Board of Regents grant NSF(2010)-PFUND-171}

\begin{abstract}
We find finite presentations for the automorphism group of the Artin pure braid group and the automorphism group of the pure braid group associated to the full monomial group.
\end{abstract}

\maketitle

\section{Introduction}

Let $B_n$ be the Artin braid group, with generators $\sigma_1,\dots,\sigma_{n-1}$ and relations $\sigma_i\sigma_{i+1}\sigma_i=\sigma_{i+1}\sigma_i\sigma_{i+1}$ for $1\le i\le n-2$, and $\sigma_{j}\sigma_i=\sigma_{i}\sigma_j$ for $|j-i|\ge 2$.  
It is well known from work of Dyer and Grossman \cite{DG81} that the automorphism group of the braid group may be realized as $\Aut(B_n)\cong\bar{B}_n \rtimes \Z_2$, where $\bar{B}_n$ denotes $B_n$ modulo its center and $\Z_2$ acts by taking generators to their inverses. 
In this paper, we find an explicit presentation for the automorphism group of the Artin pure braid, the kernel $P_n=\ker(B_n \to \Sigma_n)$ of the natural map from the braid group to the symmetric group. The pure braid group has generators
\begin{equation} \label{eq:Pn gens}
A_{i,j}=\sigma_{j-1}^{}\cdots\sigma_{i+1}^{}\sigma_i^{2}\sigma_{i+1}^{-1}\cdots\sigma_{j-1}^{-1}
=\sigma_{i}^{-1}\cdots\sigma_{j-2}^{-1}\sigma_{j-1}^{2}\sigma_{j-2}^{}\cdots\sigma_{i}^{},
\end{equation}
and relations
\begin{equation} \label{eq:purebraidrels}
A_{r,s}^{-1}A_{i,j}^{}A_{r,s}^{}=
\begin{cases}
A_{i,j}^{}&\text{if $i<r<s<j$,}\\
A_{i,j}^{}&\text{if $r<s<i<j$,}\\
A_{r,j}^{}A_{i,j}^{}A_{r,j}^{-1}&\text{if $r<s=i<j$,}\\
A_{r,j}^{}A_{s,j}^{}A_{i,j}^{}A_{s,j}^{-1}A_{r,j}^{-1}&\text{if $r=i<s<j$,}\\
[A_{r,j}^{},A_{s,j}^{}]A_{i,j}^{}[A_{r,j}^{},A_{s,j}^{}]^{-1}&\text{if $r<i<s<j$,}
\end{cases}
\end{equation}
where $[u,v]=uvu^{-1}v^{-1}$ is the commutator.  Birman \cite{Bir75} is a general reference.

Let
\[
F(\C,n)=\{(x_1,\dots,x_n) \in \C^n \mid x_i \neq x_j\ \text{if}\ i \neq j\}
\]
be the configuration space of $n$ distinct ordered points in $\C$.  The symmetric group $\Sigma_n$ acts freely on $F(\C,n)$ by permuting coordinates.  Let $C(\C,n)=F(\C,n)/\Sigma_n$ denote the orbit space, the configuration space of $n$ distinct unordered points in $\C$.  It is well known that $P_n=\pi_1(F(\C,n))$, $B_n=\pi_1(C(\C,n))$, and that these spaces are Eilenberg-Mac\,Lane spaces for these braid groups.

For $i\neq j$, let $H_{i,j}=\ker(x_i-x_j)$, and let $\A_n=\{H_{i,j}\mid 1\le i<j\le n\}$ denote the braid arrangement in $\C^n$, consisting of the reflecting hyperplanes of the symmetric group $\Sigma_n$.
The configuration space $F(\C,n)=M(\A_n)=\C^n \smallsetminus \bigcup_{1\le i<j \le n}H_{i,j}$ may be realized as the complement of the braid arrangement $\A_n$. 
The other pure braid groups we consider may be viewed as arising from an analogous construction.

Let $r$ be a natural number greater than or equal to $2$. 
The complex hyperplane arrangement $\A_{r,n}$ in $\C^n$ defined by the polynomial
\begin{equation} \label{eq:monopoly}
Q_{r,n}=Q(\A_{r,n})=x_1\cdots x_n \prod_{1\le i<j \le n} (x_i^r - x_j^r)
\end{equation}
is known as a full monomial arrangement.  The complement $M(\A_{r,n}) = \C^n \setminus Q_{r,n}^{-1}(0)$ may be realized as the orbit configuration space\[
F_{\Gamma}(\C^*,n) = \{(x_1,\dots,x_n) \in (\C^*)^{n} \mid \Gamma\cdot x_i \cap \Gamma\cdot x_j
=\emptyset\ \text{if}\ i\neq j\}
\]
of ordered $n$-tuples of points in $\C^*$ which lie in distinct orbits of the free action of $\Gamma=\Z/r\Z$ on $\C^*$ by multiplication by the primitive $r$-th root of unity $\exp(2\pi\sqrt{-1}/r)$.

Let $B(r,n)$ denote the group with generators 
$\rho_0,\rho_1,\dots,\rho_{n-1}$ and relations 
\begin{equation} \label{eq:MonomialBraidRels}
(\rho_0\rho_1)^2=(\rho_1\rho_0)^2\!,\   
\rho_i\rho_{i+1}\rho_i=\rho_{i+1}\rho_i\rho_{i+1}\, (1\le 
i<n),\ \rho_i\rho_j=\rho_j\rho_i\  (|j-i|>1).
\end{equation}
This is the (full) monomial braid 
group, the fundamental group of the orbit space 
$M(\A_{r,n})/W$, where $W=G(r,n)$ is the full monomial group, 
cf.~\cite{BMR}.  Note that $B(r,n)$ is independent of $r$.  This group 
admits a natural surjection to $G(r,n)$, which may be presented with 
generators $\rho_0,\rho_1,\dots,\rho_{n-1}$ and relations 
\eqref{eq:MonomialBraidRels} together with 
$\rho_0^r=\rho_1^2=\dots =\rho_{n-1}^2=1$.
Note that the hyperplanes of 
$\A_{r,n}$ are the reflecting hyperplanes of the 
group $G(r,n)$, 
and that this group is isomorphic to the wreath product of the symmetric group $\Sigma_n$ and the cyclic group $\Z/r\Z$.

The fundamental group of the complement $M(\A_{r,n})$ of the full monomial arrangement is the kernel $P({r,n})=\pi_1(M(\A_{r,n}))$ of the aforementioned surjection 
$B(r,n) \twoheadrightarrow G(r,n)$, which we refer to as the pure monomial braid group.  Furthermore, $M(\A_{r,n})$ is an Eilenberg-Mac\,Lane space for the pure monomial braid group. 
A presentation for the group $P({r,n})$ is found in \cite[Thm.~2.2.4]{Coh01} (in slightly different notation).\footnote{\,There is a typographical error in the second family of relations recorded in \cite[(2.9)]{Coh01}.  These relations should read $[A_{j,k}^{(p)},\,A_{j,l}^{[p]}A_{i,l}^{(q)}(A_{j,l}^{[p]})^{-1}]$.}  
For $1\le i \le n$, let 
$X_i=\rho_{i-1}\cdots\rho_2\rho_1\rho_0\rho_1\rho_2 \cdots\rho_{i-1}$, 
and define
\begin{equation} \label{eq:PureMonomialGens}
\begin{split} 
C_j^{}&=\rho_{j-1}^{}\cdots\rho_2^{}\rho_1^{}\rho_0^r
\rho_1^{-1}\rho_2^{-1}\cdots\rho_{j-1}^{-1}
\ (1\le j\le n),\\
A_{i,j}^{(q)}&=X_i^{q-r} \cdot\rho_{j-1}^{}\cdots\rho_{i+1}^{} 
\rho_i^2
\rho_{i+1}^{-1}\cdots\rho_{j-1}^{-1}\cdot X_i^{r-q}
\ (1\le i<j\le n,\ 1\le q\le r).
\end{split}
\end{equation}
These elements generate the pure monomial braid group.

Setting $r=1$ in \eqref{eq:monopoly} yields a polynomial $Q_{1,n}$ which defines an arrangement $\A_{1,n}$ whose complement has the homotopy type of the complement of the braid arrangement $\A_{n+1}$ in $\C^{n+1}$, $M(\A_{1,n}) \simeq M(\A_{n+1})$.  For $r \ge 2$, the mapping $M(\A_{r,n}) \to M(\A_{1,n})$ defined by $(x_1,\dots,x_n) \mapsto (x_1^r,\dots,x_n^r)$ is a finite covering (equivalent to the pullback along the inclusion $M(\A_{1,n}) \hookrightarrow (\C^*)^n$ of the covering $(\C^*)^n \to (\C^*)^n$ defined by the same formula).  Thus, $P(r,n)=\pi_1(M(\A_{r,n}))$ is a finite index subgroup of $P_{n+1}=\pi_1(M(\A_{1,n}))$. 

In this paper, 
building on work of Bell and Margalit \cite{BM07} and Charney and Crisp \cite{CC05},  
we find finite presentations of the automorphism groups of the pure braid groups $P_n$ and $P(r,n)$.  These automorphism groups, $\Aut(P_n)$ in particular, are used in \cite{CFR11} to study the residual freeness of these pure braid groups. 
The structure of the automorphism groups of the full braid groups $B_n$ and $B(r,n)$ is known, 
see \cite{DG81} and \cite{CC05}. The monomial braid group $B(r,n)=B(2,n)$ may be realized as the Artin group of type B, and the automorphism group $\Aut(B(2,n))$ was determined in \cite{CC05} from this perspective.

\section{Preliminaries}

In this section, we gather a number of facts regarding split extensions, (pure) braid groups, and mapping class groups which will be of use in analyzing the automorphism groups of the pure braid groups $P_n$ and $P(r,n)$.

Let $K$ be a group with trivial center, $Z(K)=1$, and let $A$ be an abelian group.  As noted by Leininger and Margalit \cite{LM06}, a split central extension
\[
1\to A \to G \leftrightarrows K \to 1
\]
induces a split extension
\begin{equation} \label{eq:splitaut}
1\to \tv(G) \to \Aut(G) \leftrightarrows \Aut(K) \to 1,
\end{equation}
where $\tv(G) < \Aut(G)$ is the subgroup consisting of all automorphisms of $G$ which become trivial upon passing to the quotient $K$.  If, moreover, $G=A \times K$ is a direct product, an explicit splitting in \eqref{eq:splitaut} is given by sending $\alpha \in \Aut(K)$ to $\tilde\alpha \in \Aut(G)$, where $\left.\tilde\alpha\right|_A=\id_A$ and $\left.\tilde\alpha\right|_K=\alpha$.  We occassionally abuse notation and write simply $\alpha$ in place of $\tilde\alpha$ in this situation.

For a group $G$ with infinite cyclic center $Z=\langle z \rangle$, a transvection is an endomorphism of $G$ of the form $x \mapsto x z^{t(x)}$, where $t\colon G \to \Z$ is a homomorphism, see Charney and Crisp \cite{CC05}.  Such a map is an automorphism if an only if its restriction to $Z$ is surjective, which is the case if and only if $z \mapsto z$ or $z \mapsto z^{-1}$, that is, $t(z)=0$ or $t(z)=-2$.  For the groups we are interested in, the extension $1\to Z(G) \to G \to G/Z(G) \to 1$ is split (in fact $G \cong Z(G) \times G/Z(G)$), and $Z(G)$ is infinite cyclic.  In this instance, the subgroup $\tv(G) < \Aut(G)$ consists of all transvection automorphisms of $G$, so we refer to $\tv(G)$ as the transvection subgroup of $\Aut(G)$.

As alluded to in the previous paragraph, the pure braid groups $P_n$ and $P(r,n)$ admit direct product decompositions
\begin{equation} \label{eq:direct}
P_n \cong Z(P_n) \times P_n/Z(P_n)\quad \text{and} \quad P(r,n) \cong Z(P(r,n)) \times P(r,n)/Z(P(r,n)),
\end{equation}
and the center of each of these groups is infinite cyclic.  The above direct product decompositions (the first of which, for $P_n$, is well known) may be obtained using results from the theory of hyperplane arrangements.  See Orlik and Terao \cite{OT} as a general reference. First, if $\A$ is a central arrangement in $\C^n$ (the hyperplanes of which all contain the origin), the restriction of the Hopf bundle $\C^n\smallsetminus\{0\} \to \CP^{n-1}$ to the complement $M=M(\A)$ yields a homeomorphism $M \cong \C^* \times \bar{M}$, where $\bar{M}$ is the complement of the projectivization of $\A$ in $ \CP^{n-1}$. Thus, $\pi_1(M) \cong \Z \times \pi_1(\bar{M})$. Second, the braid arrangement $\A_n$ and the full monomial arrangement $\A_{r,n}$ are 
fiber-type (or supersolvable) arrangements.  As such, the fundamental groups of the complements decompose as iterated semidirect products of free groups,
\begin{equation*} \label{eq:iterated}
P_n=\pi_1(M(\A_n))= \rtimes_{k=1}^{n-1} F_k
\ \text{and}\ 
P(r,n)=\pi_1(M(\A_{r,n}))=\rtimes_{k=1}^{n-1} F_{r(k-1)+1},
\end{equation*}
where $F_k$ is the free group of rank $k$.  The direct product decompositions \eqref{eq:direct} follow easily from these two facts.  Note also that these considerations imply that the groups $\bar{P}_n=P_n/Z(P_n)$ and $\bar{P}(r,n)=P(r,n)/Z(P(r,n))$ are centerless.

Explicit generators for the centers of the braid groups $B_n$, $P_n$, $B(r,n)$, and $P(r,n)$ are known.  Regarding the Artin braid groups, it is a classical result of Chow (see \cite[Cor.~1.8.4]{Bir75}) that $Z(B_n)=Z(P_n)=\Z$, generated by
\[
Z_n = (\sigma_1\sigma_2\cdots\sigma_{n-1})^n=(A_{1,2})(A_{1,3}A_{2,3})\cdots (A_{1,n}\cdots A_{n-1,n}).
\]
The centers of the monomial braid groups $B(r,n)$ and $P(r,n)$ were determined by Brou\'e, Malle, and Rouquire \cite[Prop.~3.10]{BMR}.  In terms of the generators $\rho_0,\rho_1,\dots,\rho_{n-1}$ of $B(r,n)$, these centers are given by
\[
Z(B(r,n)) = \langle(\rho_0\rho_1\cdots \rho_{n-1})^n\rangle \ \text{and}\ 
Z(P(r,n)) = \langle(\rho_0\rho_1\cdots \rho_{n-1})^{nr}\rangle.
\]
Write $\zeta_n=(\rho_0\rho_1\cdots \rho_{n-1})^n$ so that $Z(B(r,n)) = \langle\zeta_n\rangle$ and 
$Z(P(r,n)) = \langle\zeta_n^r\rangle$.  
Since $B(r,n)=B(2,n)$ is the type B Artin group, the fact that $Z(B(r,n))=\langle \zeta_n\rangle$ follows from work of Deligne \cite{Del}, see also Brieskorn and Saito \cite{BS72}.

We express $\zeta_n^r$ in terms of the generators of the pure monomial braid group $P(r,n)$ recorded in \eqref{eq:PureMonomialGens}.  For $1\le i < j \le n$, let
\begin{align} \label{eq:particularmonobraids}
A_{i,j}^{[q]}&=A_{i,j}^{(q)}A_{i,j}^{(q+1)}\cdots A_{i,j}^{(r-1)} \quad \text{(for $q<r$),}\notag\\
V_{i,j}^{(q)}&=A_{i,j}^{(q)}A_{i+1,j}^{(q)}\cdots A_{j-1,j}^{(q)}\quad \text{(for $q\le r$),}\\
D_k&=A_{k-1,k}^{[1]}A_{k-2,k}^{[1]}\cdots A_{1,k}^{[1]} C_k V_{1,k}^{(r)}\quad \text{(for $k\le n$)}.\notag
\end{align}

\begin{lemma} \label{lem:monocenter}
The center of the pure monomial braid group $P(r,n)$ is generated~by
\[
\zeta_n^r = D_1 D_2 \cdots D_n.
\]
\end{lemma}
\begin{proof}
Recall the braids $X_i=\rho_{i-1}\cdots\rho_2\rho_1\rho_0\rho_1\rho_2 \cdots\rho_{i-1}$ in $B(r,n)$.  An inductive argument using the monomial braid relations \eqref{eq:MonomialBraidRels} reveals that
\begin{equation*} \label{ZBrnX}
\zeta_n = X_1 X_2 \cdots X_n = \zeta_{n-1} \cdot X_n.
\end{equation*}
The relations \eqref{eq:MonomialBraidRels} may also be used to check that $X_iX_j=X_jX_i$ for each $i$ and $j$.  
Thus, $Z(P(r,n))$ is generated by $\zeta_n^r = (X_1 X_2 \cdots X_n)^r=X_1^r X_2^r \cdots X_n^r=\zeta_{n-1}^r \cdot X_n^r$.
We may inductively assume that $\zeta_{n-1}^r=D_1 D_2 \cdots D_{n-1}$, so it suffices to show that $D_n=X_n^r$.

Use \eqref{eq:PureMonomialGens} and \eqref{eq:particularmonobraids} to check that $C_n V_{1,n}^{(r)}=\rho_{n-1} \cdots \rho_1 \rho_0^r \rho_1 \cdots \rho_{n-1}$ and also that
$A_{i,n}^{[1]} = X_i^{1-r}(A_{i,n}^{(r)} X_i)^{r-1}$.  Then, a calculation reveals that
\[
D_n=(X_1 \cdots X_{n-1})^{1-r} Y_{n-1}^{r-1} \rho_{n-1} \cdots Y_1^{r-1} \rho_1 \rho_0^r \rho_1 \cdots \rho_{n-1},
\]
where $Y_i = \rho_i^2 X_i$.  Since $Y_i \rho_i=\rho_iX_{i+1}$ and $X_j \rho_i=\rho_i X_j$ for $i<j$, we have
\[
\begin{aligned}
D_n&=(X_1 \cdots X_{n-1})^{1-r} \rho_{n-1} \cdots \rho_1 X_n^{r-1} \cdots X_2^{r-1} \rho_0^r \rho_1 \cdots \rho_{n-1}\\
\ \qquad\qquad&=\zeta_{n-1}^{1-r} \rho_{n-1} \cdots \rho_1 \zeta_n^{r-1} \rho_0\rho_1\cdots \rho_{n-1}=\zeta_{n-1}^{1-r} \zeta_n^{r-1} X_n  = X_n^r. \qquad\qquad\ \qedhere
\end{aligned}
\]
\end{proof}

Recall that $\bar{P}_n=P_n/Z(P_n)$ and $\bar{P}(r,n)=P(r,n)/Z(P(r,n))$.  These groups may be realized as finite index subgroups of the (extended) mapping class group of the punctured sphere.  Let $\bS_m$ denote the sphere $S^2$ with $m$ punctures, and let $\Mod(\bS_{m})$ be the extended mapping class group of $\bS_{m}$, the group of isotopy classes of all self-diffeomorphisms of $\bS_{m}$.  

The mapping class group $M(0,m)$ of isotopy classes of orientation-preserving self-diffeomorphisms of $\bS_{m}$ is an index two subgroup of $\Mod(\bS_{m})$.  For $m\ge 2$, the group $M(0,m)$ admits a presentation with generators $\omega_1,\dots,\omega_{m-1}$ and relations
\begin{equation}
\begin{array}{ll} \label{eq:mcgrels}
\omega_i\omega_j=\omega_j\omega_i\ \text{for}\ |i-j|\ge 2, 
&\omega_i \omega_{i+1} \omega_i= \omega_{i+1} \omega_i \omega_{i+1}, \\
\omega_1\cdots \omega_{m-2}\omega_{m-1}^2\omega_{m-2}\cdots \omega_1=1,
&(\omega_1 \omega_2 \cdots \omega_{m-1})^{m}=1,
\end{array}
\end{equation}
see \cite[Thm.~4.5]{Bir75}. 
The extended mapping class group $\Mod(\bS_{m})$ then admits a presentation with the above generators and relations, along with the additional generator $\epsilon$ and relations $(\epsilon \omega_i)^2=1$ and $\epsilon^2=1$.

If $G$ is a subgroup of a group $\Gamma$, recall that the normalizer of $G$ in $\Gamma$ is
$N_\Gamma(G) =\{\gamma \in \Gamma \mid \gamma^{-1} G \gamma = G\}$, 
the largest subgroup of $\Gamma$ having $G$ as a normal subgroup.  Building on work of Korkmaz \cite{Kor99} and Ivanov \cite{Iv03}, Charney and Crisp \cite[Cor.~4 (ii)]{CC05} establish the following.

\begin{proposition} \label{prop:CC}
If $m \ge 5$ and $G$ is a finite index subgroup of $\Gamma=\Mod(\bS_m)$, then $\Aut(G) \cong N_\Gamma(G)$.
\end{proposition}

Throughout the paper, $\Aut(G)$ denotes the group of right automorphisms of $G$, with multiplication $\alpha\cdot \beta=\beta\circ\alpha$.

\section{Automorphisms of the Artin pure braid group}

The map $B_n \to M(0,n+1) \to \Mod(\bS_{n+1})$ given by $\sigma_i \mapsto \omega_i$, $1\le i \le n-1$, realizes $\bar{B}_n=B_n/Z$ as a finite index subgroup of the extended mapping class group $\Mod(\bS_{n+1})$, where $Z=Z(B_n)=Z(P_n)$, see, for instance, \cite{CC05}.  This comes from realizing $B_n$ as the orientation-preserving mappping class group of ${\mathbb D}_n$, the $n$-punctured disk, relative to the boundary, and including ${\mathbb D}_n$ in $\bS_{n+1}$.  In this way, $\bar{P}_n=P_n/Z$ is realized as $\mathrm{P}\!\Mod(\bS_{n+1})$, the subgroup of orientation-preserving mapping classes which fix every puncture.

The subgroup $\bar{P}_n$ is normal in $\Mod(\bS_{n+1})$.  Thus, $\Aut(\bar{P}_n) \cong \Mod(\bS_{n+1})$ for $n\ge 4$, see Proposition \ref{prop:CC}.  This fact was originally established by Korkmaz \cite{Kor99}, and extended by Bell and Margalit \cite{BM07}.  Since $P_n \cong Z \times \bar{P}_n$, the split extension \eqref{eq:splitaut} yields a semidirect product decomposition
$\Aut(P_n) \cong \tv(P_n) \rtimes \Aut(\bar{P}_n)$.  This is an ingredient in the identification, for $n\ge 4$, of the automorphism group of the pure braid group as 
\begin{equation} \label{eq:semi}
\Aut(P_n) \cong (\Z^N \rtimes \Z_2) \rtimes \Mod(\bS_{n+1})
\end{equation}
made by Bell and Margalit \cite[Thm.~8]{BM07}.  
Here, $\tv(P_n)\cong \Z^N \rtimes \Z_2$, where $N=\binom{n}{2}-1$.

Recall that the center $Z=Z(P_n)$ of the pure braid group is infinite cyclic, generated by 
$Z_n=A_{1,2}A_{1,3}A_{2,3}\cdots A_{1,n}\cdots A_{n-1,n}$.
The transvection subgroup $\tv(P_n)$ of $\Aut(P_n)$ consists of automorphisms of the form $A_{i,j} \mapsto A_{i,j} Z_n^{t_{i,j}}$, where $t_{i,j} \in \Z$ and $\sum t_{i,j}$ is either equal to $0$ or $-2$.  In the former case, $Z_n \mapsto Z_n$, while $Z_n \mapsto Z_n^{-1}$ in the latter.  This yields a surjection $\tv(P_n) \to \Z_2$, with kernel consisting of transvections for which $\sum t_{i,j}=0$.  Since $P_n$ has $\binom{n}{2}=N+1$ generators, this kernel is free abelian of rank $N$.  The choice $t_{1,2}=-2$ and all other $t_{i,j}=0$ gives a splitting $\Z_2 \to \tv(P_n)$.  Thus, 
$\tv(P_n) \cong \Z^N \rtimes \Z_2$.  This group is generated by transvections $\psi,\phi_{i,j}\colon P_n \to P_n$, $1\le i<j\le n$, $\{i,j\}\neq \{1,2\}$, where
\begin{equation} \label{eq:transvections}
\psi\colon A_{p,q} \mapsto \begin{cases} A_{1,2}Z_n^{-2}&\text{$p=1$, $q=2$,}\\ 
A_{p,q}&\text{otherwise,} \end{cases} \ 
\phi_{i,j}\colon A_{p,q} \mapsto \begin{cases} A_{1,2}Z_n^{}&\text{$p=1$, $q=2$,}\\ 
A_{i,j}Z_n^{-1}&\text{$p=i$, $q=j$,}\\ A_{p,q}&\text{otherwise.} \end{cases}
\end{equation}
It is readily checked that $\psi^2=1$ and that $\psi \phi_{i,j}\psi = \phi_{i,j}^{-1}$.  Observe that nontrivial elements of $\tv(P_n)$ are outer automorphisms.

The mapping class group $\Mod(\bS_{n+1})$ acts on $\bar{P}_n=P_n/\langle Z_n\rangle \cong \mathrm{P}\!\Mod(\bS_{n+1})$ by conjugation.  
We exhibit automorphisms of $P_n$ which fix the generator $Z_n$ of the center and induce the corresponding automorphisms of $\bar{P}_n$ upon passing to the quotient.  
For group elements $x$ and $y$, write $y^x=x^{-1}yx$.  

Define elements $\omega_k$, $1\le k \le n$, and $\epsilon$ of $\Aut(P_n)$ as follows:
\begin{equation} \label{eq:mcgautos}
\begin{aligned}
\omega_k\colon A_{i,j}&\mapsto \begin{cases}
A_{i-1,j}&\text{if $k=i-1$,}\\
A_{i+1,j}^{A_{i,i+1}}&\text{if $k=i<j-1$,}\\
A_{i,j-1}&\text{if $k=j-1>i$,}\\
A_{i,j+1}^{A_{j,j+1}}&\text{if $k=j$,}\\
A_{i,j}&\text{otherwise,}
\end{cases}
\qquad \text{for $1\le k \le n-1$, $k \neq 2$,}\\
\omega_2\colon A_{i,j}&\mapsto \begin{cases}
A_{1,3}^{A_{2,3}}Z_n&\text{if $i=1$, $j=2$,}\\
A_{1,2}Z_n^{-1}&\text{if $i=1$, $j=3$,}\\
A_{3,j}^{A_{2,3}}&\text{if $i=2$, $j\ge 4$,}\\
A_{2,j}&\text{if $i=3$,}\\
A_{i,j}&\text{otherwise,}
\end{cases}\\
\omega_n\colon A_{i,j}&\mapsto \begin{cases}
A_{i,j}&\text{if $j \neq n$,}\\
(A_{1,n}A_{1,2}A_{1,3}\cdots A_{1,n-1})^{-1}Z_n&\text{if $i=1$, $j=n$,}\\
(A_{2,n}A_{1,2}A_{2,3}\cdots A_{2,n-1})^{-1}Z_n&\text{if $i=2$, $j=n$,}\\
(A_{i,n}A_{1,i}\cdots A_{i-1,i} A_{i,i+1}\cdots A_{i,n-1})^{-1}&\text{if $3\le i$, $j=n$,}
\end{cases}\\
\epsilon\colon A_{i,j}&\mapsto 
\begin{cases} 
A_{1,2}^{-1} Z_n^2&\text{if $i=1$, $j=2$,}\\
(A_{i+1,j} \cdots A_{j-1,j})^{-1} A_{i,j}^{-1} (A_{i+1,j} \cdots A_{j-1,j})&\text{otherwise.}
\end{cases}
\end{aligned}
\end{equation}
Check that $\omega_k(Z_n)=Z_n$ for each $k$, and $\epsilon(Z_n)=Z_n$.  Also, note that, for $1\le k \le n-1$ and $k\neq 2$, the automorphism $\omega_k$ is given by the usual conjugation action of the braid $\sigma_k$ on the pure braid group, $\omega_k(A_{i,j}^{})=A_{i,j}^{\sigma_k}=\sigma_k^{-1}A_{i,j}^{}\sigma_k^{}$, see \cite{DG81}.  The automorphism $\omega_2$ is the composite of the conjugation action of $\sigma_2$ and the transvection $\phi_{1,3}$, see \eqref{eq:transvections}.  This accounts for the fact that $A_{1,2}=[(A_{1,3}A_{2,3})\cdots (A_{1,n}\cdots A_{n-1,n})]^{-1}$ in $\bar{P}_n$, the fact that, for instance, $A_{1,3}^{\sigma_2}=A_{1,2}$ in $P_n$, and insures that $\omega_2(Z_n)=Z_n$.

Similar considerations explain the occurrence of $Z_n$ in the formulas for the automorphisms $\omega_n$ and $\epsilon$ above.  The former automorphism of $P_n$ lifts the automorphism of $\bar{P}_n$ given by conjugation by $\omega_n \in \Mod(\bS_{n+1})$.  This conjugation action can be determined using the mapping class group relations \eqref{eq:mcgrels}, noting that the relations $\omega_i\omega_{i+1}\omega_i=\omega_{i+1}\omega_i\omega_{i+1}$ and  
$\omega_1\cdots \omega_{n-1}\omega_n^2\omega_{n-1}\cdots \omega_1=1$ imply that, for instance,
\[
\begin{aligned}
\omega_n^{-1} A_{n-1,n}^{} \omega_n^{}&=
\omega_n^{-1} \omega_{n-1}^{2} \omega_n^{}=
\omega_{n-1}^{} \omega_{n}^{2} \omega_{n-1}^{-1}\\
&=\omega_{n-1}^{}\left(\omega_{n-1}^{}\cdots \omega_2^{} \omega_1^2 \omega_2^{} \cdots \omega_{n-1}^{}\right)^{-1} \omega_{n-1}^{-1}\\
&=\omega_{n-1}^{}\left(A_{1,n}^{} A_{2,n}^{}  \cdots A_{n-1,n}^{}\right)^{-1} \omega_{n-1}^{-1}\\
&=(A_{n-1,n}A_{1,n-1}\cdots A_{n-2,n-1})^{-1}.
\end{aligned}
\]
Similarly, the fact that $A_{i,n}=A_{n-1,n}^{\omega_{n-2}\cdots \omega_i}$ for $i \le n-2$ may be used to calculate $\omega_n^{-1} A_{i,n}^{} \omega_n^{}$.

\begin{proposition} \label{prop:mcgrels}
The elements $\omega_1,\dots,\omega_n,\epsilon \in \Aut(P_n)$ 
satisfy the mapping class group relations \eqref{eq:mcgrels} and the relations $\epsilon^2=1$ and $(\epsilon \omega_k)^2=1$ for each $k$, $1\le k \le n$.
\end{proposition}
\begin{proof}
As noted above, for $1\le k \le n-1$ and $k\neq 2$, the automorphism $\omega_k$ is given by the conjugation action of the braid $\sigma_k$, 
$\omega_k(A_{i,j}) = A_{i,j}^{\sigma_k}$.  It follows that all of the (braid) relations \eqref{eq:mcgrels} that do not involve $\omega_2$ or $\omega_n$ hold.  
So it remains to check that the automorphisms $\omega_k$ of $P_n$ satisfy
$\omega_i\omega_2\omega_i=\omega_2\omega_i\omega_2$ for $i=1,3$, $\omega_2\omega_i=\omega_i\omega_2$ for $i \ge 4$, 
$\omega_{n-1}\omega_n\omega_{n-1}=\omega_{n}\omega_{n-1}\omega_{n}$, $\omega_i\omega_n=\omega_n\omega_i$ for $i\le n-2$, 
$\omega_1\cdots \omega_{n-1}\omega_n^2\omega_{n-1}\cdots \omega_1=1$, and 
$(\omega_1 \omega_2 \cdots \omega_n)^{n+1}=1$.  We will check the last two, and leave the others as exercises for the reader.

To verify that $\omega_1\cdots \omega_{n-1}\omega_n^2\omega_{n-1}\cdots \omega_1=1$, first check that
\begin{equation*} \label{eq:composites}
\begin{aligned}
\omega_1\cdots \omega_{n-1}(A_{i,j})&=\begin{cases}
Z_n A_{1,n}^{A_{2,n}\cdots A_{n-1,n}}&\text{if $i=1$, $j=2$,}\\
Z_n^{-1}A_{1,2}&\text{if $i=2$, $j=3$,}\\
A_{j-1,n}^{A_{j,n}\cdots A_{n-1,n}(A_{1,j-1}\cdots A_{j-2,j-1})^{-1}}&\text{if $i=1$, $j\ge 3$,}\\
A_{i-1,j-1}&\text{otherwise,}
\end{cases}\\
\omega_n^2(A_{i,j})&=\begin{cases}
A_{i,j}&\text{if $j \le n-1$,}\\
A_{i,n}^{A_{1,i}\cdots A_{i-1,i}A_{i,i+1}\cdots A_{i,n-1}}\hskip 31pt &\text{if $j=n$,}
\end{cases}\\
\omega_{n-1}\cdots \omega_{1}(A_{i,j})&=\begin{cases}
A_{2,3} Z_n \hskip 108pt &\text{if $i=1$, $j=2$,}\\
A_{1,2} Z_n^{-1}&\text{if $i=1$, $j=n$,}\\
A_{1,i+1}&\text{if $i\ge 2$, $j=n$,}\\
A_{i+1,j+1}&\text{otherwise.}
\end{cases}
\end{aligned}
\end{equation*}
These calculations, together with the pure braid relations \eqref{eq:purebraidrels}, can be used to check that 
$\omega_1\cdots \omega_{n-1}\omega_n^2\omega_{n-1}\cdots \omega_1=1$.
 
To verify that $(\omega_1 \omega_2 \cdots \omega_n)^{n+1}=1$, first note that the relations 
$\omega_i\omega_{i+1}\omega_i=\omega_{i+1}\omega_i\omega_{i+1}$ for $1\le i\le n$ and $\omega_j\omega_i=\omega_i\omega_j$ for $|j-i|\ge 2$ imply that 
\[
\begin{aligned}
(\omega_1 \omega_2 \cdots \omega_n)^{n+1}&=(\omega_1 \omega_2 \cdots \omega_{n-1})^{n}\cdot \omega_n\cdots \omega_2
\omega_1^2\omega_2\cdots \omega_n,\ \text{and}\\
\omega_n\cdots \omega_2\omega_1^2\omega_2\cdots \omega_n&=(\omega_n\cdots \omega_1)\cdot 
\omega_1\cdots \omega_{n-1}\omega_n^2\omega_{n-1}\cdots \omega_1 (\omega_n\cdots \omega_1)^{-1}.
\end{aligned}
\]
Since $\omega_1\cdots \omega_{n-1}\omega_n^2\omega_{n-1}\cdots \omega_1=1$ by the previous paragraph, it suffices to check that 
$(\omega_1 \omega_2 \cdots \omega_{n-1})^{n}=1$.

Write $\tau=\omega_1\cdots \omega_{n-1}$.  We must show that $\tau^n=1$. The action of $\tau$ on the pure braid generators $A_{i,j}$ is given above.  In particular, $\tau(A_{i,j})=A_{i-1,j-1}$ for $i\ge 2$ and $j\ge 4$.  Also, note that for $j \ge 3$, the pure braid relations \eqref{eq:purebraidrels} may be used to show that 
\[
\tau(A_{1,j})=A_{j-1,n}^{A_{j,n}\cdots A_{n-1,n}(A_{1,j-1}\cdots A_{j-2,j-1})^{-1}}
=A_{j-1,n}^{A_{1,n}\cdots A_{n-1,n}}.
\]

Observe that $\tau^{n-3}(A_{n-1,n})=A_{2,3}$.  Consequently, $\tau^{n-2}(A_{n-1,n})=Z_n^{-1}A_{1,2}$, and $\tau^{n-1}(A_{n-1,n})=A_{1,n}^{A_{2,n}\cdots A_{n-1,n}}$. 
A calculation then reveals that $\tau^n(A_{n-1,n})=A_{n-1,n}$. It follows that $\tau^{n-k}(A_{k-1,k})=A_{n-1,n}$ for $k\ge 3$, which implies that $\tau^n(A_{k-1,k})=A_{k-1,k}$ for $k\ge 3$.

If $i=j-k$ with $k\ge 2$ (so that $j\ge 3$), then $A_{i,j}=A_{j-k,j}=\tau^{n-j}(A_{n-k,n})$.  If $\tau^j(A_{i,j})=A_{n-k,n}$, it follows that $\tau^n(A_{n-k,n})=A_{n-k,n}$ and then that $\tau^n(A_{i,j})=A_{i,j}$. Thus, it suffices to show that $\tau^j(A_{i,j})=A_{n-k,n}$. 
If $i>1$, then $\tau^{i-1}(A_{i,j})=A_{1,j-i+1}$.  So it is enough to show that $\tau^q(A_{1,q})=A_{n-q+1,n}$, where $q \ge 3$. Checking that
\[
\tau^p(A_{1,q})=A_{q-p,n-p+1}^{A_{1,n-p+1}\cdots A_{n-p,n-p+1} \cdot A_{n-p+1,n-p+2}\cdots A_{n-p+1,n}}
\]
for $1\le p \le q-1$, we have
\[
\begin{aligned}
\tau^q(A_{1,q})&=\tau(A_{1,n-q}^{A_{2,n-q}\cdots A_{n-q-1,n-q} \cdot A_{n-q,n-q+1}\cdots A_{n-q,n}})\\
&=A_{n-q-1,n}^{A_{n-q,n}\cdots A_{n-1,n} \cdot A_{n-q-1,n-q}\cdots A_{n-q-1,n}}
\end{aligned}
\]
A calculation with the pure braid relations \eqref{eq:purebraidrels} then shows that $\tau^q(A_{1,q})=A_{n-q+1,n}$.

It remains to check that $\epsilon^2=1$ and $(\epsilon \omega_k)^2=1$ for each $k$, $1\le k \le n$.  The first of these is straightforward.  For the remaining ones, note that 
$\epsilon(A_{i,j} \cdots A_{j-1,j})=
(A_{i,j} \cdots A_{j-1,j})^{-1}$ for $i>1$ and $j>2$, $\epsilon(A_{i,i+1}\cdots A_{i,j})=(A_{i,i+1}\cdots A_{i,j})^{-1}$ for $i>1$,  while $\epsilon(A_{1,2}\cdots A_{1,j})=(A_{1,2}\cdots A_{1,j})^{-1}Z_n^2$.  These observations, together with the pure braid relations \eqref{eq:purebraidrels} may be used to verify that $\omega_k \epsilon \omega_k = \epsilon$ for each $k$, $1\le k \le n$.
\end{proof}

Thus, the elements $\omega_1,\dots,\omega_n$ and $\epsilon$ of $\Aut(P_n)$ satisfy the relations of the extended mapping class group $\Mod(\bS_{n+1})$. By construction, these elements of $\Aut(P_n)$ induce the automorphisms of 
$\bar{P}_n \cong \mathrm{P}\!\Mod(\bS_{n+1})$ corresponding to conjugation by the generators (with the same names) of $\Mod(\bS_{n+1})$ upon passing to the quotient.

\begin{theorem} \label{thm:autpres}
For $n\ge 4$, the automorphism group $\Aut(P_n)$ of the pure braid group admits a presentation with generators
\[
\epsilon, \ \omega_k,\ 1\le k \le n,\ \psi,\ \phi_{i,j}, \ 1\le i < j \le n,\ \{i,j\}\neq \{1,2\},
\]
and relations
\begin{equation*}
\begin{aligned}
&\begin{matrix} 
\omega_i\omega_j=\omega_j\omega_i,\ |i-j|\ge 2, \hfill
&\omega_i \omega_{i+1} \omega_i= \omega_{i+1} \omega_i \omega_{i+1},\  i< n, 
&\epsilon^2=1, \hfill\\[3pt]
(\omega_1 \omega_2 \cdots \omega_n)^{n+1}=1,\hfill 
&\omega_1\cdots \omega_{n-1}\omega_n^2\omega_{n-1}\cdots \omega_1=1,\hfill
&(\epsilon \omega_k)^2=1,\  k \le n,\\[3pt]
\psi\phi_{i,j}^{} \psi=\phi_{i,j}^{-1},\ \forall i,j, \hfill
&\phi_{i,j}\phi_{p,q}=\phi_{p,q}\phi_{i,j},\ \forall i,j,p,q\hfill
&\psi^2=1,\hfill \\[3pt]
\epsilon\psi \epsilon=\psi, \hfill
&\omega_k^{-1}\psi^{}_{}\omega_k^{}=\psi,\ k\le n \hfill
&\epsilon\phi_{i,j}^{} \epsilon=\phi_{i,j}^{-1},  \hfill
\end{matrix}
\end{aligned}
\end{equation*}
\begin{equation*}
\begin{aligned}
&\begin{matrix}
\omega_1^{-1}\phi_{i,j}^{}\omega_1^{}=\begin{cases}
\phi_{2,j}&i=1,\\ \phi_{1,j}&i=2,\\ 
\phi_{i,j}&\text{otherwise,}
\end{cases}\qquad \qquad
&
\omega_2^{-1}\phi_{i,j}^{}\omega_2^{}=\begin{cases}
\phi_{1,3}^{-1}&i=1,j=3,\\ 
\phi_{1,3}^{-1}\phi_{3,j}^{}&i=2, j> 3,\\ 
\phi_{1,3}^{-1}\phi_{2,j}^{}&i= 3,\\ 
\phi_{1,3}^{-1}\phi_{i,j}^{}&\text{otherwise,}
\end{cases} 
\end{matrix}\\
&
\omega_k^{-1}\phi_{i,j}^{}\omega_k^{}=\begin{cases}
\phi_{i-1,j}&k=i-1,\\ \phi_{i+1,j}&k=i<j-1,\\
\phi_{i,j-1}&k=j-1>i,\\ \phi_{i,j+1}&k=j,\\ 
\phi_{i,j}&\text{otherwise,}
\end{cases}\quad \text{for $3\le k \le n-1$,}\\
&
\omega_n^{-1}\phi_{i,j}^{}\omega_n^{}=\begin{cases}
\phi_{i,j}{}\phi_{1,n}^{}\phi_{2,n}^{}\phi_{i,n}^{-1}\phi_{j,n}^{-1}&j<n,\\ 
\phi_{i,n}^{-1}\phi_{1,n}^{}\phi_{2,n}^{}&j=n.
\end{cases}
\end{aligned}
\end{equation*}
\end{theorem}
\begin{proof}
Recall from \eqref{eq:splitaut} and \eqref{eq:semi} that there is a split, short exact sequence
\[
1\to \tv(P_n) \to \Aut(P_n) \leftrightarrows \Mod(\bS_{n+1}) \to 1.
\]
Since the automorphisms $\psi$ and $\phi_{i,j}$ generate the transvection subgroup $\tv(P_n)$, and the automorphisms $\epsilon$ and $\omega_k$ induce the generators of $\Mod(\bS_{n+1})=\Aut(\bar{P}_n)$, these automorphisms collectively generate $\Aut(P_n)$.  By Proposition \ref{prop:mcgrels}, the automorphisms $\epsilon$ and $\omega_k$ satisfy the extended mapping class group relations.  As noted previously, the formulas \eqref{eq:transvections} may be used to show that the transvections $\psi$ and $\phi_{i,j}$ satisfy $\psi^2=1$ and $\psi\phi_{i,j}\psi=\phi_{i,j}^{-1}$.  So it suffices to show that the actions of the automorphisms $\epsilon$ and $\omega_k$ on the transvections $\psi$ and $\phi_{i,j}$ are as asserted. This may be accomplished by calculations with the explicit descriptions of these automorphisms given in \eqref{eq:transvections} and \eqref{eq:mcgautos}.
\end{proof}

\begin{remark} Theorem \ref{thm:autpres} exhibits the semidirect product structure of $\Aut(P_n) \cong \tv(P_n) \rtimes \Mod(\bS_{n+1})$.  Recall that $\Mod(\bS_{n+1})=M(0,n+1) \rtimes \Z_2$ is itself the semidirect product of the (non-extended) mapping class group and $\Z_2$.  Note that the generator $\psi$ of $\tv(P_n) < \Aut(P_n)$ commutes with the generators $\epsilon,\omega_1,\dots,\omega_n$ of $\Aut(P_n)$ which induce the generators of $\Mod(\bS_{n+1})$.  It follows that $\Aut(P_n)$ may be realized as the iterated semidirect product $\Aut(P_n) \cong (\Z^N \rtimes M(0,n+1)) \rtimes (\Z_2 \times \Z_2)$.
\end{remark}

Similar considerations yield a 
presentation for the automorphism group of the three strand pure braid group $P_3\cong \Z \times F_2$.  In this case, the split extension \eqref{eq:splitaut} yields $\Aut(P_3) \cong \tv(P_3) \rtimes \Aut(F_2)$, 
where $\tv(P_3)\cong \Z^2 \rtimes \Z_2$, generated by $\psi,\phi_{1,3},\phi_{2,3}$ 
with $\psi^2=1$ and $\psi\phi_{i,j}^{}\psi=\phi_{i,j}^{-1}$ 
(see \eqref{eq:transvections}), and 
\[
F_2 = \bar{P}_3=P_3/Z(P_3)=P_3/\langle A_{1,2}A_{1,3}A_{2,3}\rangle = \langle A_{1,3}, A_{2,3} \rangle
\]
is the free group on two generators.  The group $\Aut(F_2)$ admits the following presentation, due to Neumann (see \cite[\S3.5, Prob.~2]{MKS}):
\[
\Aut(F_2) = \langle P,\sigma,U \mid P^2,\sigma^2,(\sigma P)^4,(P\sigma PU)^2,(UP\sigma)^3,[U,\sigma U \sigma]\rangle,
\]
where the automorphisms $P,\sigma,U$ of $F_2$ are given by
\[
\begin{matrix}
P(A_{1,3}^{})=A_{2,3}, \quad & \sigma(A_{1,3}^{})=A_{1,3}^{-1}, \quad & U(A_{1,3}^{})=A_{1,3}^{}A_{2,3}^{},\\
P(A_{2,3}^{})=A_{1,3},\hfill & \sigma(A_{2,3}^{})=A_{2,3}^{},\hfill & U(A_{2,3}^{})=A_{2,3}^{}.\hfill
\end{matrix}
\]
Lifts of these automorphisms to automorphisms of $P_3$ fixing $Z_3=A_{1,2}A_{1,3}A_{2,3}$ are given by setting 
\[
P(A_{1,2}^{})=A_{2,3}A_{1,2}A_{2,3}^{-1},\quad
\sigma(A_{1,2}^{})=A_{1,2}^{}A_{1,3}^2,\quad
U(A_{1,2}^{})=A_{2,3}^{-1}A_{1,2}.
\]
Calculations with these formulas yield the following result.
\begin{proposition}
The automorphism group $\Aut(P_3)$ of the three strand pure braid group admits a presentation with generators $P,\ \sigma,\ U,\ \psi,\ \phi_{1,3},\ \phi_{2,3}$,
and relations 
\[
\begin{matrix}
P^2, & \sigma^2, & (\sigma P)^4, & (P\sigma PU)^2, & (UP\sigma)^3, & [U,\sigma U \sigma],\\
[U,\psi], & [P,\psi], & [\sigma,\psi], & [U,\phi_{1,3}],& P\phi_{1,3}P\phi_{2,3}^{-1}, & (\sigma\phi_{1,3})^2, \\
 \psi^2, & (\psi\phi_{i,3})^2, & [\phi_{1,3},\phi_{2,3}], & \phi_{1,3}[\phi_{2,3}^{},U],& P\phi_{2,3}^{}P\phi_{1,3}^{-1}, & [\sigma,\phi_{2,3}^{}].
\end{matrix}
\]
\end{proposition}
\begin{remark} Note that the generator $\psi$ of $\tv(P_3)< \Aut(P_3)$ commutes with the generators $P,\sigma,U$ of $\Aut(P_3)$ which project to the generators of $\Aut(F_2)$.  It follows that 
$\Aut(P_3) \cong \Z^2 \rtimes (\Z_2 \times \Aut(F_2))$.
\end{remark}

Since the two strand pure braid group $P_2=\Z$ is infinite cyclic, $\Aut(P_2)=\Z_2$.

\section{Automorphisms of the pure monomial braid group}
As discussed for example in \cite[\S3]{BMR}, the full monomial braid group $B(r,n)=B(2,n)$ embeds in the Artin braid group $B_{n+1}$.  
In terms of the standard generators $\sigma_i$, $1\le i \le n$, of $B_{n+1}$ and the generators $\rho_j$, $0\le j \le n-1$ of $B(r,n)$, one choice of embedding is 
given by $\rho_0 \mapsto \sigma_1^2$ and $\rho_j \mapsto \sigma_{j+1}$ for $j\neq 0$. 
Restricting to the pure monomial braid group yields a monomorphism $P(r,n) \to P_{n+1}$.  In terms of the generators \eqref{eq:Pn gens} of $P_{n+1}$ and \eqref{eq:PureMonomialGens} of $P(r,n)$, this is given by
\[
C_j\mapsto A_{1,j+1}^r,\quad
A_{i,j}^{(q)}\mapsto (A_{1,i+1}  \cdots A_{i,i+1})^{q-r} A_{i+1,j+1}
(A_{1,i+1}  \cdots A_{i,i+1})^{r-q}.
\]
Recall the generators $Z_{n+1}=(\sigma_1\cdots\sigma_n)^{n+1}$ and $\zeta_n=(\rho_0\cdots\rho_{n-1})^n$ of the centers $Z(B_{n+1})=Z(P_{n+1})$ and $Z(B(r,n))$, and that $Z(P(r,n))$ is generated by $\zeta_n^r$.  It is readily checked that the above embedding takes $\zeta_n$ to $Z_{n+1}$.  Consequently, the group $\bar{P}(r,n)=P(r,n)/Z(P(r,n))$ may be realized as a (finite index) subgroup of $\bar{P}_{n+1}=P_{n+1}/Z(P_{n+1})$.

Composing with the map $B_{n+1} \to \Mod(\bS_{n+2})$ given by $\sigma_i \mapsto \omega_i$ realizes $\bar{P}(r,n)$ as a finite index subgroup of the extended mapping class group $\Gamma=\Mod(\bS_{n+2})$.  Hence, for $n \ge 3$, we have $\Aut(\bar{P}(r,n)) \cong N_\Gamma(\bar{P}(r,n))$ by Proposition \ref{prop:CC}.  
Since $P(r,n) \cong Z(P(r,n)) \times \bar{P}(r,n)$, the split extension \eqref{eq:splitaut} yields a semidirect product decomposition
$\Aut(P(r,n)) \cong \tv(P(r,n)) \rtimes \Aut(\bar{P}(r,n))$.  Thus, for $n\ge 3$, the automorphism group of the pure monomial braid group may be realized as
\begin{equation*} \label{eq:monosemi}
\Aut(P(r,n)) \cong \tv(P(r,n)) \rtimes N_\Gamma(\bar{P}(r,n)).
\end{equation*}

\begin{lemma} \label{lem:monotrans}
Let $N_r=r\binom{n}{2}+n-1$. 
The transvection subgroup of the automorphism group of the pure monomial braid group is given by 
$\tv(P(r,n)) \cong \Z^{N_r} \rtimes \Z_2$, 
where $\Z_2$ acts on $\Z^{N_r}$ by taking elements to their inverses.
\end{lemma}
\begin{proof}
For notational convenience, denote the generator of the center of $P(r,n)$ by $Z_{r,n}=\zeta_n^r$.  
In terms of the generators \eqref{eq:PureMonomialGens} of $P(r,n)$, the transvection subgroup $\tv(P(r,n))$ of $\Aut(P(r,n))$ consists of automorphisms of the form 
$C_j \mapsto C_j Z_{r,n}^{s_j}$ and $A_{i,j}^{(q)} \mapsto A_{i,j}^{(q)} Z_{r,n}^{t_{i,j,q}}$,
where $s_j,t_{i,j,q}\in\Z$ and $S=\sum_{j=1}^n s_j + \sum_{q=1}^r\sum_{1\le i<j\le n} t_{i,j,q}$ is either equal to $0$ or $-2$.  In the former case, $Z_{r,n} \mapsto Z_{r,n}$, while $Z_{r,n} \mapsto Z_{r,n}^{-1}$ in the latter.  This yields a surjection $\tv(P(r,n)) \to \Z_2$, with kernel consisting of transvections for which $ S=0$.  Since $P(r,n)$ has $N_r+1$ generators, this kernel is free abelian of rank $N_r$.  Setting $s_{1}=-2$, $s_j=0$ for $2\le j \le n$, and all $t_{i,j,q}=0$ gives a splitting $\Z_2 \to \tv(P(r,n))$.  Thus, 
$\tv(P(r,n)) \cong \Z^{N_r} \rtimes \Z_2$.  This group is generated by transvections $\Psi$, $\Upsilon_i$, $2\le i \le n$, $\Phi_{i,j,p}$, $1\le i <j\le n$, $1\le p \le r$, of $P(r,n)$, 
defined by
\begin{align} \label{eq:monotransvections}
\Psi\colon &\begin{cases} 
C_j \mapsto C_1Z_{r,n}^{-2}&\text{if $j= 1$,}\\ 
C_j \mapsto C_j&\text{if $j\neq 1$,}\\ 
A_{k,l}^{(q)}\mapsto A_{k,l}^{(q)}&\text{for all $k$, $l$, $q$,} \end{cases} \quad
\Upsilon_i\colon \begin{cases} 
C_j \mapsto C_1Z_{r,n}&\text{if $j= 1$,}\\ 
C_j \mapsto C_iZ_{r,n}^{-1}&\text{if $j=i$,}\\ 
C_j \mapsto C_j,&\text{if $j\neq 1,i$,}\\ 
A_{k,l}^{(q)}\mapsto A_{k,l}^{(q)}&\text{for all $k$, $l$, $q$,} \end{cases} \\
\Phi_{i,j,p}\colon &\begin{cases} 
C_j \mapsto C_1Z_{r,n}&\text{if $j= 1$,}\\ 
C_j \mapsto C_j&\text{if $j\neq 1$,}\\ 
A_{k,l}^{(q)}\mapsto A_{i,j}^{(p)} Z_{r,n}^{-1}&\text{if $k=i$, $l=j$, $q=p$,}\\
A_{k,l}^{(q)}\mapsto A_{k,l}^{(q)}&\text{otherwise.} \end{cases} \notag 
\end{align}
Check that the transvections $\Upsilon_i$, $\Phi_{i,j,p}$ all commute, and that $\Psi^2=1, \Psi \Upsilon_{i}\Psi = \Upsilon^{-1}_{i}$, and $\Psi \Phi_{i,j,p}\Psi = \Phi_{i,j,p}^{-1}$  to complete the proof.
\end{proof}

For $n\ge 3$, viewing the group $\bar{P}(r,n)$ as a subgroup of the extended mapping class group via the sequence of embeddings
\[
\bar{P}(r,n) \to \bar{P}_{n+1} \to \bar{B}_{n+1} \to M(0,n+2) \to \Mod(\bS_{n+2}),
\]
the group $\Mod(\bS_{n+2})$ acts on $\bar{P}(r,n)$ by conjugation.  The subgroup $\bar{P}(r,n)<
\Mod(\bS_{n+2})$ is, however, not a normal subgroup.  For instance, one can check that $\omega_1\cdot 
\bar{P}(r,n)\neq \bar{P}(r,n)\cdot\omega_1$.  Thus, the normalizer $N_\Gamma(\bar{P}(r,n))$ of $\bar{P}(r,n)$ in $\Gamma=\Mod(\bS_{n+2})$ is a proper subgroup of $\Mod(\bS_{n+2})$.

So to understand the structure of $\Aut(P(r,n))=\tv(P(r,n)) \rtimes N_\Gamma(\bar{P}(r,n))$, we must determine this normalizer.  For $n\ge 3$, the normalizer $N_\Gamma(\bar{B}(r,n))$ of $\bar{B}(r,n)=\bar{B}(2,n)=B(2,n)/Z(B(2,n))$ in $\Gamma=\Mod(\bS_{n+2})$ was found by Charney and Crisp 
\cite[Prop.~10]{CC05}:
\begin{equation}\label{eq:Bnormalizer}
N_\Gamma(\bar{B}(r,n)) \cong \bar{B}(r,n))\rtimes (\Z_2 \times \Z_2).
\end{equation}
Identifying the generators of $\bar{B}(r,n)$ with their images in $\Mod(\bS_{n+2})$, the group $N_\Gamma(\bar{B}(r,n))$ has generators
$\rho_0=\omega_1^2, \rho_1=\omega_2,\dots,\rho_{n-1}=\omega_n, \epsilon, \Delta$,
where 
\[
\Delta = \omega_1 \cdots \omega_{n+1}\cdot \omega_1 \cdots \omega_n\cdot \omega_1 \cdots \omega_{n-1} \cdots\cdots \omega_1\cdot\omega_2\cdot\omega_1
\]
in $\Mod(\bS_{n+2})$. Note that $\Delta^2=(\omega_1\cdots\omega_{n+1})^{n+2}=1$.  
The elements  $\epsilon$ and $\Delta$ generate $\Z_2\times\Z_2$.  Their action on $\bar{B}(r,n))$ is given by $\epsilon\colon\rho_i\mapsto \rho_i^{-1}$ and 
\begin{equation*} \label{eq:Deltaaction}
\Delta\colon \rho_i \mapsto
\begin{cases}
(\rho_{n-1} \cdots \rho_1 \rho_0 \rho_1 \cdots \rho_{n-1})^{-1}&\text{if $i=0$,}\\
\rho_{n-i}&\text{if $1\le i \le n-1$.}
\end{cases}
\end{equation*}

\begin{proposition} \label{prop:Pnormalizer}
Let $\Gamma=\Mod(\bS_{n+2})$. For $n \ge 3$, $N_\Gamma(\bar{P}(r,n))=N_\Gamma(\bar{B}(r,n))$.
\end{proposition}
\begin{proof}
Since $\bar{P}(r,n)$ is normal in $\bar{B}(r,n)$, we have $\rho_i(\bar{P}(r,n))=\bar{P}(r,n)$ for each $i$, $0\le i \le n-1$.  It is straightforward to check that 
$\epsilon(\bar{P}(r,n))=\bar{P}(r,n)$.  We assert that $\Delta(\bar{P}(r,n))=\bar{P}(r,n)$ as well, which would imply that $\bar{P}(r,n)$ is normal in $N_\Gamma(\bar{B}(r,n))$.

For this, recall the monomial braids $X_i=\rho_{i-1}\cdots\rho_1\rho_0\rho_1\cdots\rho_{i-1}$, and note that $\Delta(\rho_0)=X_n^{-1}$ and more generally, $\Delta(X_i)=X_{n-i+1}^{-1}$.  Recall also from the proof of Lemma \ref{lem:monocenter} that $X_i^r=D_i$ is a pure monomial braid.  Using these observations, one can check (on the generators of $\bar{P}(r,n)$, see \eqref{eq:PureMonomialGens}) that $\Delta(\bar{P}(r,n))=\bar{P}(r,n)$.  Thus, $\bar{P}(r,n)\triangleleft N_\Gamma(\bar{B}(r,n))$.

The above considerations imply that $N_\Gamma(\bar{B}(r,n))$ is a subgroup of $N_\Gamma(\bar{P}(r,n))$, since the latter is the largest subgroup of $\Mod(\bS_{n+2})$ in which $\bar{P}(r,n)$ is normal.  However, the (right) cosets of $H=N_\Gamma(\bar{B}(r,n))$ in $\Gamma=\Mod(\bS_{n+2})$ are $H$ and $H\cdot\omega_1$, and since 
$\omega_1\cdot \bar{P}(r,n)\neq \bar{P}(r,n)\cdot\omega_1$, the same is true for any element of $H\cdot\omega_1$.  It follows that 
$N_\Gamma(\bar{P}(r,n))=N_\Gamma(\bar{B}(r,n))$.
\end{proof}

Hence we have $\Aut(P(r,n)) \cong \tv(P(r,n)) \rtimes N_\Gamma(\bar{B}(r,n))$,
and we now turn our attention to exhibiting a presentation for this group.  As done with the Artin pure braid group in the previous section, we exhibit automorphisms of $P(r,n)$ which fix the generator $Z_{r,n}=\zeta_n^r$ of the center, and induce the corresponding (conjugation) automorphisms of $\bar{P}(r,n)$ upon passing to the quotient.  

The automorphisms $\epsilon$ and $\Delta$ of $\bar{B}(r,n)$ extend to automorphisms of $B(r,n)$ 
(denoted by the same symbols) which take the generator $\zeta_n$ of the center $Z(B(r,n))$ to its inverse. 
For $\beta \in B(r,n)$, let $c_\beta\in\Aut(P(r,n))$ be the automorphism given by conjugation by $\beta$, $c_\beta(x)=\beta^{-1}x\beta$.  Recall the transvection automorphisms $\Psi$, $\Upsilon_i$, $\Phi_{i,j,p}$ of $P(r,n)$ defined in \eqref{eq:monotransvections}, and 
define elements $\tilde\rho_k$, $0\le k \le n-1$, $\tilde\epsilon$, and $\tilde\Delta$ of $\Aut(P(r,n))$ as follows:
\begin{equation} \label{eq:monoauts}
\tilde\rho_0 = c_{\rho_0},\ \tilde\rho_1= c_{\rho_1} \circ \Upsilon_2,\ \tilde\rho_k=c_{\rho_k}\ (2\le k \le n-1),\ \tilde\epsilon= \epsilon \circ \Psi,\ \tilde\Delta = \Delta \circ \Psi \circ \Upsilon_n.
\end{equation}
Since $c_\beta(Z_{r,n})=Z_{r,n}$, 
$\epsilon(Z_{r,n})=Z_{r,n}^{-1}$, 
$\Delta(Z_{r,n})=Z_{r,n}^{-1}$, $\Upsilon_j(Z_{r,n})=Z_{r,n}$, and $\Psi(Z_{r,n})=Z_{r,n}^{-1}$, each of the automorphisms defined above fixes $Z_{r,n}$.  
Explicit formulas for the actions of these automorphisms on the pure monomial braid generators \eqref{eq:PureMonomialGens} may be obtained through calculations using  
the monomial braid relations \eqref{eq:MonomialBraidRels} and the presentation for $P(r,n)$ found in \cite[Thm.~2.2.4]{Coh01} (see also \cite[Lem.~2.2.3]{Coh01}).  The results of these calculations are relegated to the next section.

\begin{proposition} \label{prop:Nrels}
The automorphisms $\tilde\rho_0,\dots,\tilde\rho_{n-1},\tilde\epsilon,\tilde\Delta\in\Aut(P(r,n))$ 
satisfy 
\begin{equation*}
\begin{array}{lll} 
\tilde\rho_i \tilde\rho_{i+1} \tilde\rho_i= \tilde\rho_{i+1} \tilde\rho_i \tilde\rho_{i+1}\ \text{for $1\le i<n$}, 
&\tilde\rho_i\tilde\rho_j=\tilde\rho_j\tilde\rho_i\ \text{for}\ |i-j|\ge 2, 
&\tilde\epsilon^2=1,\\[2pt]
(\tilde\rho_0 \tilde\rho_1 \cdots \tilde\rho_{n-1})^{n}=1,
&(\tilde\rho_0\tilde\rho_1)^2=(\tilde\rho_1\tilde\rho_0)^2,
&\tilde\Delta^2=1,\\[2pt]
\tilde\Delta\tilde\rho_0\tilde\Delta=(\tilde\rho_{n-1}\cdots\tilde\rho_1\tilde\rho_0\tilde\rho_1\cdots\tilde\rho_{n-1})^{-1},
&(\tilde\epsilon\tilde\rho_k)^2=1\ \text{for}\ 0\le k< n,
&[\tilde\epsilon,\tilde\Delta]=1, \\[2pt]  
\tilde\Delta\tilde\rho_k\tilde\Delta=\tilde\rho_{n-k}\ \text{for}\ 1\le k< n.
\end{array}
\end{equation*}
\end{proposition}
These are the relations of the normalizer of $\bar{P}(r,n)$ in $\Mod(\bS_{n+2})$.
\begin{proof}[Sketch of proof]
Since $\tilde\rho_k$ is conjugation by $\rho_k$ for $k \neq 1$, all of the relations which do not involve $\tilde\rho_1$, $\tilde\epsilon$, and $\tilde\Delta$ hold since they hold in the monomial braid group.  Additionally, note that the automorphisms $\rho_k$, $\epsilon$, and $\Delta$ of $\bar{P}(r,n)$ generate the normalizer 
$N_\Gamma(\bar{P}(r,n))=N_\Gamma(\bar{B}(r,n))$, where $\Gamma=\Mod(\bS_{n+2})$, so they satisfy the analogs of the relations stated in the Proposition.

These observations, together with the formulas for the automorphisms $\tilde\rho_k$, $\tilde\epsilon$, $\tilde\Delta$, $\Psi$, and $\Upsilon_i$ recorded in \S\ref{sec:monoaction} and \eqref{eq:monotransvections}, may be used to verify that all of the asserted relations hold.  For instance, let $\tau= \rho_0\rho_1\cdots\rho_{n-1}$ and $\tilde\tau=\tilde\rho_0\tilde\rho_1\cdots\tilde\rho_{n-1}$.  Note that $\tau^n=1$.  One can check that 
\[
\begin{array}{ll}
\tilde\tau(C_1)=\tau(C_1) \cdot Z_{r,n} = C_n^{W_1}\cdot Z_{r,n}, & 
\tilde\tau(C_2)=\tau(C_2) \cdot Z_{r,n}^{-1} = C_1^{W_2}\cdot Z_{r,n}^{-1},\\  
\tilde\tau(C_j)=\tau(C_j) = C_{j-1}^{W_j} \ (3\le j\le n), &
\tilde\tau(A_{i,j}^{(q)}) = \tau(A_{i,j}^{(q)}),
\end{array}
\] 
for certain words $W_j\in P(r,n)$.  This, together with the fact $\tau^n=1$, may be used to show that $\tilde\tau^n=(\tilde\rho_0 \tilde\rho_1 \cdots \tilde\rho_{n-1})^{n}=1$.

For the relation $\tilde\Delta\tilde\rho_0\tilde\Delta=(\tilde\rho_{n-1}\cdots\tilde\rho_1\tilde\rho_0\tilde\rho_1\cdots\tilde\rho_{n-1})^{-1}$, it is enough to show that
$\tilde\lambda=\tilde\rho_{n-1}\cdots\tilde\rho_1\tilde\rho_0\tilde\rho_1\cdots\tilde\rho_{n-1}\tilde\Delta\tilde\rho_0\tilde\Delta=1$.  The analogous automorphism $\lambda=\rho_{n-1}\cdots\rho_1\rho_0\rho_1\cdots\rho_{n-1}\Delta\rho_0\Delta$ is trivial (consider its action on the generators of $B(r,n)$).  Checking that $\tilde\lambda(x)=\lambda(x)$ for each generator $x$ of $P(r,n)$ reveals that $\tilde\lambda=1$ as well.

Verification of the remaining relations may be handled in a similar manner, and is left to the reader.
\end{proof}

\begin{theorem} \label{thm:monoautpres}
For $n\ge 3$, the automorphism group $\Aut(P(r,n))$ of the pure monomial braid group admits a presentation with generators
\[
\tilde\epsilon,\ \tilde\Delta,\ \tilde\rho_k,\ 0\le k \le n-1,\ \Psi,\ \Upsilon_l, \ 2\le l \le n,\ \Phi_{i,j,p}, \  1\le i < j \le n,\  1\le p\le r,
\]
and relations
\begin{align*}
&\begin{array}{lll} 
\tilde\rho_i \tilde\rho_{i+1} \tilde\rho_i= \tilde\rho_{i+1} \tilde\rho_i \tilde\rho_{i+1}\ \text{for $1\le i<n$}, 
&\tilde\rho_i\tilde\rho_j=\tilde\rho_j\tilde\rho_i\ \text{for}\ |i-j|\ge 2, 
&\tilde\epsilon^2=1,\\[2pt]
(\tilde\rho_0 \tilde\rho_1 \cdots \tilde\rho_{n-1})^{n}=1,
&(\tilde\rho_0\tilde\rho_1)^2=(\tilde\rho_1\tilde\rho_0)^2,
&\tilde\Delta^2=1,\\[2pt]
\tilde\Delta\tilde\rho_0\tilde\Delta=(\tilde\rho_{n-1}\cdots\tilde\rho_1\tilde\rho_0\tilde\rho_1\cdots\tilde\rho_{n-1})^{-1},
&(\tilde\epsilon\tilde\rho_k)^2=1\ \text{for}\ 0\le k< n,
&[\tilde\epsilon,\tilde\Delta]=1,\\[2pt]
\tilde\Delta\tilde\rho_k\tilde\Delta=\tilde\rho_{n-k}\ \text{for}\ 1\le k< n,
\end{array} \displaybreak[0]\\ 
&\ \,\Psi^2=1,\ \ [\Upsilon_l,\Phi_{i,j,p}]=1, \forall l,i,j,p,\ \ [\Phi_{i,j,p},\Phi_{k,l,q}]=1, \forall i,j,k,l,p,q,\ \ \tilde\epsilon\Psi\tilde\epsilon=\Psi,\\
&\  \,\Psi\Upsilon_l\Psi=\Upsilon_l^{-1}, \forall l,\ \ \Psi\Phi_{i,j,p}\Psi=\Phi_{i,j,p}^{-1}, \forall i,j,p,\ \ 
\ \,\tilde\Delta\Psi\tilde\Delta=\Psi,\quad 
\tilde\rho_k^{-1}\Psi\tilde\rho_k^{}=\Psi, \forall k,\displaybreak[0]\\
&\ \,\tilde\Delta\Upsilon_l\tilde\Delta=\begin{cases}\Upsilon_{n-l+1}^{-1}\Upsilon_n&l<n,\\ \Upsilon_n&l=n,\end{cases}\qquad 
\tilde\epsilon\Upsilon_l\tilde\epsilon=\Upsilon_l^{-1}, \forall l,\qquad 
\tilde\rho_0^{-1}\Upsilon_l\tilde\rho_0^{}=\Upsilon_l, \forall l,\displaybreak[0]\\
&\ \,\tilde\rho_1^{-1}\Upsilon_l^{}\tilde\rho_1^{}=\begin{cases}\Upsilon_2^{-1}&l=2,\\ \Upsilon_2^{-1}\Upsilon_l^{}&l\neq 2,\end{cases}\qquad 
\tilde\rho_k^{-1}\Upsilon_l^{}\tilde\rho_k^{}=\begin{cases}
\Upsilon_{k+1}&l=k,\\ \Upsilon_k&l=k+1,\\ \Upsilon_l&l\neq k,k+1,\end{cases}\ \text{for $k\ge 2$,} \displaybreak[0]\\
&\ \,\tilde\Delta\Phi_{i,j,p}\tilde\Delta=\begin{cases}
\Upsilon_{n-j+1}^{-1}\Upsilon_{n-i+1}^{-1}\Upsilon_n\Phi_{n-j+1,n-i+1,p}&j<n,\\
\Upsilon_{n-i+1}^{-1}\Upsilon_n\Phi_{1,n-i+1,p}&j=n, 
\end{cases} \displaybreak[0]\\
&\ \,\tilde\epsilon\Phi_{i,j,p}\tilde\epsilon=\begin{cases}
\Phi_{i,j,r-p}^{-1}& p< r,\\ \Phi_{i,j,r}^{-1}&p=r,\end{cases}
\qquad \tilde\rho_0^{-1}\Phi_{i,j,q}\tilde\rho_0^{}=\begin{cases}\Phi_{1,j,r}&i=1,q=1,\\
\Phi_{1,j,q-1}&i=1,q\neq 1,\\\Phi_{i,j,q}&\text{otherwise},\end{cases} \displaybreak[0]\\
&\ \,\tilde\rho_1^{-1}\Phi_{i,j,q}^{}\tilde\rho_1^{}=\begin{cases}\Upsilon_2^{-1}\Phi_{1,2,r-p}&i=1,j=2,p<r,\\
\Upsilon_2^{-1}\Phi_{2,j,p}^{}&i=1,j>2,\\ \Upsilon_2^{-1}\Phi_{1,j,p}^{}&i=2,\\
\Upsilon_2^{-1}\Phi_{i,j,p}^{}&\text{otherwise},\end{cases} \displaybreak[0]\\
&\ \,\tilde\rho_k^{-1}\Phi_{i,j,q}^{}\tilde\rho_k^{}=\begin{cases}
\Phi_{k,j,p}&k=i-1,\\ \Phi_{k+1,j,p}&k=i<j-1,\\ \Phi_{k,k+1,r-p}&k=i=j-1, p<r,\\ \Phi_{i,k,p}&k=j-1>i,\\ \Phi_{i,k+1,p}&k=j,\\ \Phi_{i,j,p}&\text{otherwise}
\end{cases}\qquad \text{for $k\ge 2$.}
\end{align*}
\end{theorem}
\begin{proof}
By Proposition \ref{prop:CC} and Proposition \ref{prop:Pnormalizer}, there is a split, short exact sequence
\[
1\to \tv(P(r,n)) \to \Aut(P(r,n)) \leftrightarrows N_\Gamma(\bar{P}(r,n)) \to 1,
\]
where $\Gamma=\Mod(\bS_{n+2})$. From the proof of Lemma \ref{lem:monotrans}, the automorphisms $\Psi$, $\Upsilon_l$, and $\Phi_{i,j,p}$ generate the transvection subgroup $\tv(P(r,n))$, and satisfy the relations $\Psi^2=1, [\Upsilon_l,\Phi_{i,j,p}]=1, \Psi\Upsilon_l\Psi=\Upsilon_l^{-1}, \Psi\Phi_{i,j,p}^{-1}\Psi=\Phi_{i,j,p}^{-1}$. Since the automorphisms $\tilde\epsilon, \tilde\Delta, \tilde\rho_k$ induce the generators of $N_\Gamma(\bar{P}(r,n))$, these automorphisms, together with the aforementioned transvections, generate $\Aut(P(r,n))$.  By Proposition \ref{prop:Nrels}, the automorphisms $\tilde\epsilon, \tilde\Delta, \tilde\rho_k$ satisfy the relations of $N_\Gamma(\bar{P}(r,n))$.   So it suffices to show that the actions of these automorphisms on the transvections $\Psi$, $\Upsilon_l$, and $\Phi_{i,j,p}$ are as asserted. This may be accomplished by calculations with the descriptions of these automorphisms given in \eqref{eq:monoauts}, \S\ref{sec:monoaction}, and \eqref{eq:monotransvections}.
\end{proof}

\begin{remark} Theorem \ref{thm:monoautpres} exhibits the semidirect product structure $\Aut(P(r,n)) \cong \tv(P(r,n)) \rtimes N_\Gamma(\bar{P}(r,n))$, where $\Gamma=\Mod(\bS_{n+2})$ and $\tv(P(r,n)) \cong \Z^{N_r} \rtimes \Z_2$.  Recall that $N_\Gamma(\bar{P}(r,n))=N_\Gamma(\bar{B}(r,n))=\bar{B}(r,n) \rtimes (\Z_2\times\Z_2)$ (see \eqref{eq:Bnormalizer}), and note that the transvection $\Psi$ commutes with all generators of $N_\Gamma(\bar{P}(r,n))$.  It follows that 
$\Aut(P(r,n)) \cong (\Z^{N_r} \rtimes \bar{B}(r,n))\rtimes (\Z_2\times\Z_2\times\Z_2)$.
\end{remark}

In the case $n=2$, we have $P(r,2) \cong \Z \times F_{r+1}$, and the split extension  \eqref{eq:splitaut} yields $\Aut(P(r,2)) \cong \tv(P(r,2))\rtimes \Aut(F_{r+1})$,
where $\tv(P(r,2))\cong \Z^{r+1} \rtimes \Z_2$, generated by $\Psi$, $\Upsilon_2$, $\Phi_{1,2,p}$, $1\le p \le r$,  
(see Lemma \ref{lem:monotrans}), and 
\[
F_{r+1} = \bar{P}(r,2)=P(r,2)/Z(P(r,2))=P(r,2)/\langle Z_{r,2}\rangle = \langle A_{1,2}^{(1)}, \dots, A_{1,2}^{(r-1)}, C_2, A_{1,2}^{(r)} \rangle
\]
is the free group on $r+1\ge 3$ generators.  The group $\Aut(F_{r+1})$ admits a presentation, due to Nielsen, with generators $P$, $Q$, $\sigma$, and $U$, where
\[
\begin{array}{ll}
P\colon \begin{cases} A_{1,2}^{(1)} \mapsto A_{1,2}^{(2)},\\ A_{1,2}^{(2)} \mapsto A_{1,2}^{(1)},\\
A_{1,2}^{(q)} \mapsto A_{1,2}^{(q)}&q\neq 1,2,\\ C_2 \mapsto C_2, \end{cases}
&
Q\colon \begin{cases} A_{1,2}^{(q)} \mapsto A_{1,2}^{(q+1)}&q<r-1,\\ A_{1,2}^{(r-1)} \mapsto C_2,\\
A_{1,2}^{(r)} \mapsto A_{1,2}^{(1)},\\ C_2 \mapsto A_{1,2}^{(r)}, \end{cases} \\[3pt] 
\sigma\colon \begin{cases} A_{1,2}^{(1)} \mapsto (A_{1,2}^{(1)})^{-1},\\ A_{1,2}^{(q)} \mapsto A_{1,2}^{(q)}&q\neq 1,\\ C_2 \mapsto C_2, \end{cases}
&
U\colon \begin{cases} A_{1,2}^{(1)} \mapsto A_{1,2}^{(1)}A_{1,2}^{(2)},\\ A_{1,2}^{(q)} \mapsto A_{1,2}^{(q)}&q\neq 1,\\ C_2 \mapsto C_2, \end{cases}
\end{array}
\]
Refer to \cite[\S3.5, Cor.~N1]{MKS} for the relations satisfied by these generators of $\Aut(F_{r+1})$. 
Lifts of these automorphisms to automorphisms of $P(r,2)$ which fix the generator $Z_{r,2}=C_1A_{1,2}^{(1)}\cdots A_{1,2}^{(r-1)}C_{2}A_{1,2}^{(r)}$ of the center are given by setting
\[
\begin{array}{ll}
P(C_1)=C_1[A_{1,2}^{(1)}, A_{1,2}^{(2)}],\qquad\qquad&  
Q(C_1) = (A_{1,2}^{(1)})^{-1}C_1 A_{1,2}^{(1)},\\[3pt] 
\sigma(C_1)=C_1 (A_{1,2}^{(1)})^2,& 
U(C_1)=C_1 A_{1,2}^{(1)} (A_{1,2}^{(1)}A_{1,2}^{(2)})^{-1}.
\end{array}
\]
Calculations with these formulas yields the following result.

\begin{proposition} \label{prop:r2mono}
The automorphism group of the pure monomial braid group $P(r,2)$ is generated by 
$
P,\ Q,\ \sigma,\ U,\ \Psi,\ \Upsilon_2,\ \Phi_{1,2,p},\  1\le p\le r, 
$
and decomposes as a semidirect product $\Aut(P(r,2)) \cong \tv(P(r,n)) \rtimes \Aut(F_{r+1})$.  
The action of $\Aut(F_{r+1})$ on $\tv(P(r,n))$ is given by
\[
\begin{array}{ll}
P\colon\begin{cases} \Psi \mapsto \Psi, \\ \Upsilon_2 \mapsto \Upsilon_2, \\ 
\Phi_{1,2,1} \mapsto \Phi_{1,2,2},\\ \Phi_{1,2,2} \mapsto \Phi_{1,2,1},\\ \Phi_{1,2,p} \mapsto \Phi_{1,2,p}&p\neq 1,2,\end{cases}
&
Q\colon\begin{cases} \Psi \mapsto \Psi, \\ \Upsilon_2 \mapsto \Phi_{1,2,r}, \\ 
\Phi_{1,2,p} \mapsto \Phi_{1,2,p+1}&p\le r-2,\\ \Phi_{1,2,r-1} \mapsto \Upsilon_2,\\ \Phi_{1,2,r} \mapsto \Phi_{1,2,1},\end{cases}\\ 
\sigma\colon\begin{cases} \Psi \mapsto \Psi, \\ \Upsilon_2 \mapsto \Upsilon_2, \\ 
\Phi_{1,2,1} \mapsto \Phi_{1,2,1}^{-1},\\ \Phi_{1,2,p} \mapsto \Phi_{1,2,p}&p\neq 1,\end{cases} 
&
U\colon\begin{cases} \Psi \mapsto \Psi, \\ \Upsilon_2 \mapsto \Upsilon_2, \\ 
\Phi_{1,2,2} \mapsto \Phi_{1,2,1}^{-1}\Phi_{1,2,2},\\ \Phi_{1,2,p} \mapsto \Phi_{1,2,p}&p\neq 2.\end{cases}
\end{array}
\]
\end{proposition}
\begin{remark} Note that the generator $\Psi$ of $\tv(P(r,2))< \Aut(P(r,2))$ commutes with the generators $P,Q,\sigma,U$ of $\Aut(P(r,2))$ which project to the generators of $\Aut(F_{r+1})$.  It follows that 
$\Aut(P(r,2)) \cong \Z^{r+1} \rtimes (\Z_2 \times \Aut(F_{r+1}))$.
\end{remark}

Since $P(r,1)=\Z$ is infinite cyclic, $\Aut(P(r,1))=\Z_2$.

\section{The action of $\Aut(P(r,n))$ on the generators of $P(r,n)$} \label{sec:monoaction}
We record the action of the elements $\tilde\epsilon, \tilde\Delta, \tilde\rho_k, 0\le k \le n-1$, of $\Aut(P(r,n))$ 
on the generators $C_j$, $1\le j\le n$, and $A_{i,j}^{(q)}$, $1\le i < j \le n$, $1\le q \le r$ of the pure monomial braid group $P(r,n)$.  
See \eqref{eq:monotransvections} for the action of the generators of the transvection subgroup $\tv(P(r,n))$.
Recall the elements $A_{i,j}^{[q]},  V_{i,j}^{(q)}, 
D_k \in P(r,n)$ from \eqref{eq:particularmonobraids}, and that $y^x=x^{-1}yx$. 
For $1\le i<j \le n$ and $1\le q \le r$, let $U_i^{(q)}=A_{i,i+1}^{(q)}A_{i,i+2}^{(q)} \cdots A_{i,n}^{(q)}$. 
The actions of  $\tilde\epsilon$,  $\tilde\Delta$, and $\tilde\rho_k$ are given by:
\begin{align*} \label{eq:monoauto}
\tilde\epsilon&\colon 
\begin{cases}
C_{j}\mapsto C_1^{-1} Z^2_{r,n} &\text{if $j=1$,}\\
C_{j}\mapsto (V_{1,j}^{(r)})^{-1}{C_j^{-1}}V_{1,j}^{(r)} &\text{if $j\neq 1$,}\\
A_{i,j}^{(q)}\mapsto (V_{i+1,j}^{(r)})^{-1}{(A_{i,j}^{(r)})^{-1}} V_{i+1,j}^{(r)}&\text{if $q=r$,}\\
A_{i,j}^{(q)}\mapsto (V_{i+1,j}^{(r)})^{-1}D_i^{}{(A_{i,j}^{(r-q)})^{-1}}D_i^{-1}V_{i+1,j}^{(r)}&\text{if $q\neq r$,}
\end{cases}\displaybreak[2]\\
\tilde\Delta&\colon 
\begin{cases}
C_{j}\mapsto D_n^{-1} Z_{r,n} &\text{if $j=1$,}\\
C_{j}\mapsto [U_{n-j+1}^{(r)} D_{n-j+1} U_{n-j+1}^{(1)} \cdots U_{n-j+1}^{(r-1)}]^{-1}&\text{if $j\neq 1,n$,}\\
C_{j}\mapsto [U_{1}^{(r)} D_{1} U_{1}^{(1)} \cdots U_{1}^{(r-1)}]^{-1} Z_{r,n} &\text{if $j=n$,}\\
A_{i,j}^{(q)}\mapsto (A_{n-j+1,n-i+1}^{(r)})^{V_{n-j+1,n-i+1}^{(r)}}&\text{if $q=r$,}\\
A_{i,j}^{(q)}\mapsto (A_{n-j+1,n-i+1}^{(q)})^{V_{n-j+1,n-i+1}^{(r)}D_{n-j+1}^{-1}D_{n-i+1}^{-1}}&\text{if $q\neq r$,}
\end{cases}\displaybreak[2]\\
\tilde\rho_0&\colon 
\begin{cases}
C_{j}\mapsto C_1&\text{if $j=1$,}\\
C_{j}\mapsto {C_j}^{(A_{1,j}^{(r-1)})^{-1}}&\text{if $j\neq 1$,}\\
A_{i,j}^{(q)}\mapsto A_{i,j}^{(q-1)}&\text{if $i=1$, $q\neq 1$,}\\
A_{i,j}^{(q)}\mapsto (A_{i,j}^{(r)})^{C_1}&\text{if $i=1$, $q=1$,}\\
A_{i,j}^{(q)}\mapsto A_{i,j}^{(q)}&\text{if $i\ge 2$,}
\end{cases}\displaybreak[2]\\
\tilde\rho_1&\colon 
\begin{cases}
C_{j}\mapsto C_2^{A_{1,2}^{(r)}} Z_{r,n}&\text{if $j=1$,}\\
C_{j}\mapsto C_1 Z_{r,n}^{-1}&\text{if $j=2$,}\\
C_{j}\mapsto C_j &\text{if $j\ge 3$,}\\
A_{i,j}^{(q)}\mapsto (A_{1,2}^{(r-q)})^{(C_1A_{1,2}^{(1)}\cdots A_{1,2}^{(r-q-1)})^{-1}}&\text{if $i=1$, $j=2$ $q<r$,}\\
A_{i,j}^{(q)}\mapsto A_{2,j}^{(q)}&\text{if $i=1$, $j\ge 3$ $q=r$,}\\
A_{i,j}^{(q)}\mapsto (A_{2,j}^{(q)})^{(C_1A_{1,2}^{(1)}\cdots A_{1,2}^{(r-q-1)}C_1^{-1})^{-1}}&\text{if $i=1$, $j\ge 3$ $q<r$,}\\
A_{i,j}^{(q)}\mapsto (A_{1,j}^{(q)})^{A_{1,2}^{(q+1)}\cdots A_{1,2}^{(r)}}&\text{if $i=2$, $q<r$,}\\
A_{i,j}^{(q)}\mapsto A_{i,j}^{(q)}&\text{otherwise,}
\end{cases}
\displaybreak[2]\\
\tilde\rho_k&\colon 
\begin{cases}
C_{j}\mapsto C_{j-1} &\text{if $k=j-1$,}\\
C_{j}\mapsto C_{j+1}^{A_{j,j+1}^{(r)}}&\text{if $k=j$,}\\
C_{j}\mapsto C_j &\text{if $k\neq j-1,j$,}\\
A_{i,j}^{(q)}\mapsto (A_{i-1,j}^{(q)})^{A_{i-1,i}^{[q]}A_{i-1,i}^{(r)}}&\text{if}\ k=i-1,q<r,\\
A_{i,j}^{(q)}\mapsto A_{i-1,j}^{(r)}&\text{if}\ k=i-1,q=r,\\
A_{i,j}^{(q)}\mapsto (A_{i+1,j}^{(q)})^{(D_iA_{i,i+1}^{(1)}\cdots A_{i,i+1}^{(r-q-1)}D_i^{-1})^{-1}}&\text{if}\ k=i<j-1,q<r,\\
A_{i,j}^{(q)}\mapsto A_{i+1,j}^{(r)}&\text{if}\ k=i<j-1,q=r,\\
A_{i,j}^{(q)}\mapsto (A_{i,i+1}^{(r-q)})^{(D_iA_{i,i+1}^{(1)}\cdots A_{i,i+1}^{(r-q-1)})^{-1}}&\text{if}\ k=i=j-1,q<r,\\
A_{i,j}^{(q)}\mapsto A_{i,j-1}^{(q)}&\text{if}\ k=j-1>i,\\
A_{i,j}^{(q)}\mapsto (A_{i,j+1}^{(q)})^{A_{j,j+1}^{(r)}}&\text{if}\ k=j,\\
A_{i,j}^{(q)}\mapsto A_{i,j}^{(q)}&\text{otherwise,}
\end{cases}
\end{align*}
for $2\le k \le n-1$.

\newpage

\newcommand{\arxiv}[1]{{\texttt{\href{http://arxiv.org/abs/#1}{{arXiv:#1}}}}}

\newcommand{\MRh}[1]{\href{http://www.ams.org/mathscinet-getitem?mr=#1}{MR#1}}

\bibliographystyle{amsalpha}

\end{document}